\documentclass[letterpaper]{amsart}
\usepackage{amsfonts,amssymb,amsthm,amsmath,mathrsfs,graphicx,wrapfig}

\newcommand{\A}{\mathcal{A}}

\theoremstyle{plain}
\newtheorem{theorem}{Theorem}
\newtheorem{lemma}[theorem]{Lemma}

\newtheorem{proposition}[theorem]{Proposition}
\newtheorem{corollary}[theorem]{Corollary}

\theoremstyle{remark} 

\newtheorem{remark}[theorem]{Remark}

\theoremstyle{definition} 

\newtheorem{defn}[theorem]{Definition}

\title{A Tur\'an-type Problem for Circular Arc Graphs}
\author{Rosalie Carlson, Stephen Flood, Kevin O'Neill, and Francis Edward Su}
\thanks{Supported in part by NSF Grants DMS-0701308 and DMS-1002938 (Su).}

\begin{document}
\maketitle

\begin{abstract}
A circular arc graph is the intersection graph of a collection of connected arcs on the circle.  We solve a Tur\'an-type problem for circular arc graphs: for $n$ arcs, if $m$ and $M$ are the minimum and maximum number of arcs that contain a common point, what is the maximum number of edges the circular arc graph can contain?  We establish a sharp bound and produce a maximal construction.  For a fixed $m$, this can be used to show that if the circular arc graph has enough edges, there must be a point that is covered by at least $M$ arcs.  In the case $m=0$, we recover results for interval graphs established by Abbott and Katchalski (1979).  We suggest applications to voting situations with interval or circular political spectra.
\end{abstract}

\section{Introduction}

In extremal graph theory, Tur\'an's theorem \cite{turan} says, roughly speaking, that if a graph $G$ on $n$ vertices contains enough edges, then $G$ contains a $K_r$, i.e., a complete subgraph among some set of $r$ vertices.
Here, ``enough'' means roughly more than $\frac{r-2}{r-1}$ of the ${n \choose 2}$ possible edges.  Tur\'an's result holds for general graphs, but in certain classes of graphs we may expect that we don't need quite as many edges to produce a $K_r$.

One class of graphs of great interest are interval graphs, which are intersection graphs of a collection of intervals on the line.  For interval graphs, Abbott and Katchalski \cite{abbott-katchalski} have shown that to guarantee a $K_r$, a sufficient fraction of edges is $\frac{r}{n}(2-  \frac{r}{n})$.  And since a collection of intervals on the line has the Helly property, a complete subgraph in the intersection graph corresponds to a collection of intervals that contain a common point. In many applications it is this common point that is of interest.  For instance, if each interval corresponds to a voter and represents that voter's approval of political positions along a spectrum, then such results will show the existence of a point on the political spectrum that is approved by a large number of voters(e.g. \cite{berg, hardin}).  Or, if each interval corresponds to a person's time-availability for a meeting, then such results show the existence of a time when a large number of people could meet.

We are interested in studying similar questions for a collection of intervals on a circle.  Following the literature (e.g., \cite{hardin, lin}), we call these intervals (circular) \emph{arcs} and their intersection graph $G$ a \emph{circular arc graph}.  Motivated by the examples above, the problem we study is: 
\begin{quote}
How many edges in $G$ are needed to guarantee that a given number of arcs contain a common point?
\end{quote}
However, since a collection of arcs may not in general have the Helly property, locating a large clique in $G$ is not equivalent to finding a set of arcs that contain a common point (although see \cite{lin} for an examination of sufficient conditions for a collection of circular arcs to have the Helly property).  For this reason we obtain our results by studying the combinatorics of the arcs themselves rather than finding large cliques in the intersection graph.

In fact, we construct a set of $n$ circular arcs whose intersection graph  $G$ has a maximal number of edges given every point in the circle is in at most $M$ arcs and at least $m$ arcs.  Using this maximal construction, we can prove one of our main results (Theorem~\ref{discreteEdgeCount}) which specifies a number $e_{min}$ and proves:
\begin{quote}
if $G$ has more than $e_{min}$ edges and every point is in at least $m$ arcs, then there must be point of the circle that is in at least $M$ arcs.
\end{quote}
Our other main result (Theorem \ref{continuousEdgeCount}) gives a continuous analogue in the limit when the number of arcs is asympotically large by specifying an $\alpha$ in terms of a $\beta$ and $\gamma$ such that:
\begin{quote}
if $G$ has more than proportion $\alpha$ of all possible edges and every point is in at least proportion $\gamma$ of the arcs, then there must a be point of the circle that is in at least proportion $\beta$ of the arcs.
\end{quote}
Theorem \ref{continuousEdgeCount} shows that \[
  \alpha = \left\{
  \begin{array}{l l}
    \beta(2-(1-\gamma)^2\beta) &  \quad 0 \le \beta \le \frac{1}{2}\\
    \beta(\gamma+1)(2-\beta(\gamma+1)) & \quad \frac{1}{2} < \beta \le \frac{1}{1+\gamma}.\\
  \end{array} \right.
\]

When $m=0$ (and hence $\gamma=0$), our results recover those of Abbott and Katchalski \cite{abbott-katchalski}, but by very different methods.  Theorems~\ref{discreteEdgeCount} and \ref{continuousEdgeCount} also apply in cases where $m$ is unknown; then we assume $m = M-1$ (hence $\gamma=1$) to obtain a worst-case bound.  For example, if $\beta < 1/2$ and $\gamma=1$, then $\alpha = 2\beta$, which would imply that to ensure that there is a point that is proportion $\beta=1/4$ of all arcs, it is enough to have $\alpha = 1/2$ of all possible edges appear in the intersection graph of $G$.

There is a nice voting theory intepretation.  Following \cite{berg}, we define a \emph{society} $S$ to be a triple $(X, V, \mathcal{A})$ where $X$ is a geometric space (the \emph{political spectrum}), $V$ is a set of voters, and $\mathcal{A}$ is a collection of \emph{approval sets}, one for each voter.  One usually thinks of the spectrum $X$ as a line (with conservative positions on the right and liberal positions on the left) and approval sets as intervals on the line.  However, there are many situations where the spectrum is naturally a circle \cite{golumbic-trenk, hardin}, and the approval sets are naturally circular arcs.  We may think of two overlapping approval sets as two voters ``agreeing'' on some platform.  The above result guarantees, for instance, that if 1/2 of all pairs of voters can agree on a platform, then there is a platform that at least 1/4 of all voters would simultaneously approve.  Furthermore if we know that some platform is approved by fewer voters, the fraction 1/4 can be improved.



We give an outline of our paper.  All terms used here are defined in section 2.
We want to establish a bound $e_{min}(M,m,n)$ such that if at least $e_{min}$ pairwise intersections are present in a collection of $n$ arcs $\mathcal{A}$ with minimum agreement number $m$, a maximum agreement number of at least $M$ can be guaranteed.  This bound will be derived by fixing a particular $M$ and $m$ for a collection of $n$ arcs, and determining the maximum possible number of pairwise intersections given those conditions.  Then, if there are more than that number of pairwise intersections, we will know the maximum agreement number must be at least $M+1$.

To determine the maximum number of pairwise intersections, we use a formula we derive in Section 3 that relates the number of edges in the intersection graph $G$ of $\mathcal{A}$ with the agreement number at the endpoints of $\mathcal{A}$.  Specifically, the formula provided by Theorem~\ref{edgeformula} allows the number of edges plus double intersections to be calculated from the running count sum of $\mathcal{A}$ (see Definition~4) and the number of arcs in $\mathcal{A}$.

Section 4 establishes that for fixed maximum and minimum agreement numbers, the maximum edge count as well as the maximum number of double intersections can be achieved with (possibly different) collections of arcs that maximize the running count sum (Theorem~\ref{maxatmax}).  This allows us to restrict our attention to a single $LR$-sequence with maximum running count (for fixed $M,m,n$) (Proposition~\ref{maxlr}).


Section 5 uses these results to construct a specific collection of arcs for fixed $n,M$ and $m$ that has the maximum possible number of edges.  Using Theorem~\ref{edgeformula}, the number of edges and double intersections are counted, and maximality is verified (Theorem~\ref{discreteEdgeCount}).  We also give an asymptotic result for large $n$ in Theorem~\ref{continuousEdgeCount}.  Section 6 refines this result when there is a known number of double intersections, improving the bound from Theorems~\ref{discreteEdgeCount} and \ref{continuousEdgeCount}.  We conclude with some observations about future work, including applications in voting theory.

\section{Some important definitions}
Here we introduce some of the terms and notation to be used in the rest of the paper.  Let $\mathcal{A}$ be a collection of $n$ arcs on a circle $C$.  We will assume that all endpoints in $\mathcal{A}$ are distinct.  (If two endpoints are not distinct, one of them can be moved by a small distance $\epsilon$ so that the intersections of the arcs are not changed.)    Moving clockwise around the circle, each arc $A_i \in \mathcal{A}$ has \textbf{left endpoint} $x_{i}$ and \textbf{right endpoint} $y_{i}$.  We may write $A_i=[x_i,y_i]$.

\begin{figure}[h]
  \begin{center}
     \includegraphics[scale=0.48]{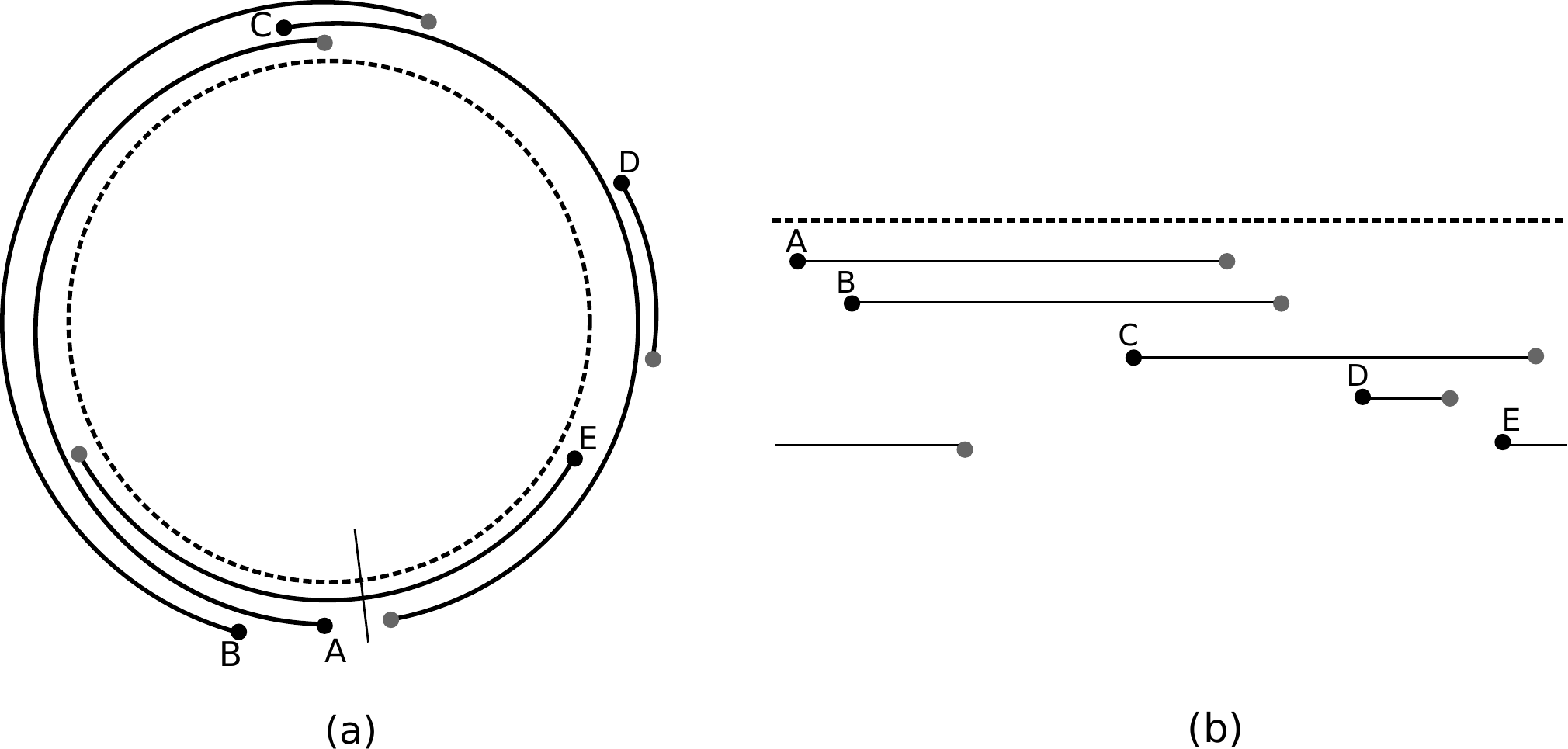}
  \end{center}
  \caption{\small{A set of circular arcs $\mathcal{A}$ represented on (a) a circle and (b) an interval.  See \cite{hardin}.}}
  \label{Def-example}
\end{figure}

\begin{defn}
Let the \textbf{intersection graph} of $\mathcal{A}$ denote the graph $G$ which represents the arcs of $\mathcal{A}$ as vertices, where two vertices are connected by an edge in $G$ if their corresponding arcs intersect in $\mathcal{A}$.  When we say that $\A$ \textbf{has $e$ edges}, we will mean that the intersection graph $G$ has $e$ edges.  We will often refer to $e$ as the \textbf{edge count} of $\A$.
\end{defn}

\begin{defn}
For $p\in C$, define $a(p)$, the \textbf{agreement number} of $p$, as the number of arcs of $\mathcal{A}$ containing $p$.  Let
\[M = \max_{p\in C} a(p)\]
and
\[m = \min_{p\in C} a(p).\]
We call $M$ the \textbf{maximum agreement number} of $\mathcal{A}$ and $m$ the \textbf{minimum agreement number} of $\mathcal{A}$.  The \textbf{agreement proportion} of $\mathcal{A}$ is defined as $\frac{M}{n}$.  The maximum agreement of the collection of arcs in Figure~\ref{Def-example} is 3, its minimum agreement is 1, and its agreement proportion is $\frac{3}{5}$.  Unless otherwise specified, $n$ will always denote the number of arcs under consideration, and $M$ and $m$ will always denote their maximum  and minimum agreement numbers.  A family of arcs will \textbf{satisfy conditions $M,m,n$} if it has $n$ arcs, maximum intersection $M$ and minimum intersection $m$.
\end{defn}

\begin{defn}
The $LR$-sequence $[q_i]$ of $\mathcal{A}$ is a sequence of symbols $L$ and $R$ derived by starting at any point, moving clockwise around the circle recording an $L$ for every left endpoint and an $R$ for every right endpoint.  $\mathcal{A}$ may have many different $LR$ sequences depending on the starting point.  If $A \in \mathcal{A}$ is an arc, let $L_A$ or $R_A$ be its left or right endpoint in the $LR$-sequence.  The set of arcs in Figure~\ref{Def-example} has the $LR$ sequence \[[L,L,R,L,R,R,L,R,L,R],\] starting at the left endpoint of arc $A$.
\end{defn}

\begin{defn}
The \textbf{running count} of an arc endpoint (or member of an $LR$ sequence) is the agreement number just to the right of that endpoint.  A \textbf{running count sequence} $[r_i]$ will consist of the corresponding running counts of a particular $LR$ sequence.  Starting at the left endpoint of arc $A$, the arcs of Figure~\ref{Def-example} have the running count sequence \[[2,3,2,3,2,1,2,1,2,1].\]  Note that the largest running count will be $M$ and the smallest running count $m$.
\end{defn}

\begin{defn}
The \textbf{running count sum} of $\mathcal{A}$ is the sum of numbers in the running count sequence. (Note that the running count sum of $\mathcal{A}$ does not depend on where the running count sequence starts.)  The arcs of Figure~\ref{Def-example} have a running count sum of 19.  A \textbf{maximal running count sequence} will refer to any running count sequence with the greatest possible running count sum for given conditions $M,m,n$.  A \textbf{maximal $LR$-sequence} is an $LR$-sequence corresponding to a maximal running count sequence.
\end{defn}

\begin{defn}
Suppose that arcs $A_1$ and $A_2$ intersect such that each arc contains both endpoints of the other, covering the entire circle between them.  We call their intersection a \textbf{double intersection}, and say that $A_1$ and $A_2$ \textbf{doubly intersect}. (See Figure~\ref{DI}.)
\end{defn}

\begin{figure}[h]
  \begin{center}
     \includegraphics[scale=0.4]{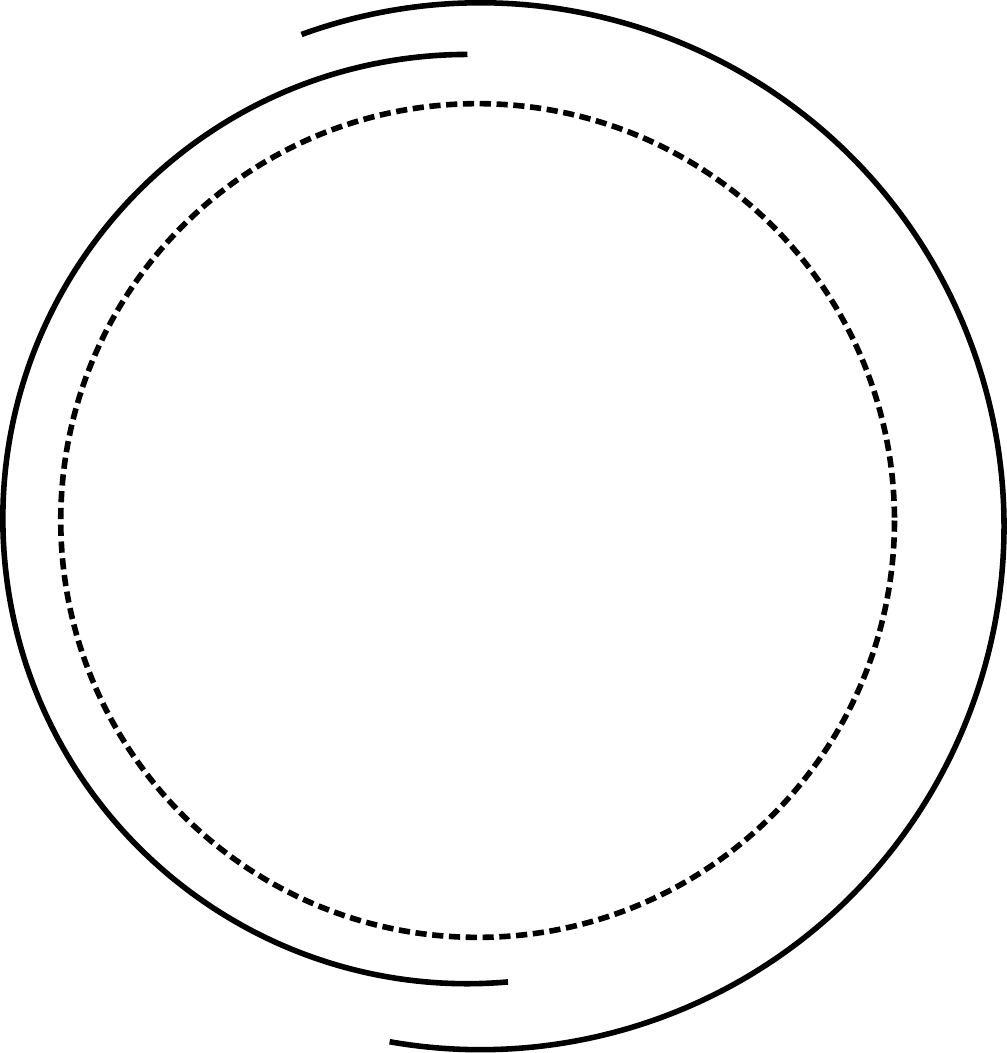}
  \end{center}
  \vspace{-10pt}
  \caption{\small{Two doubly intersecting arcs.}}
  \vspace{-10pt}
  \label{DI}
\end{figure}

\section{Running Counts and the Edge Formula}
We shall be interested in counting the edges of a collection of circular arcs, as well as finding a point covered by a large number of arcs.  The following result provides a helpful connection between these goals.

\begin{theorem}[Edge Formula]\label{edgeformula}
Given a collection $\mathcal{A}$ of $n$ circular arcs, with $e$ edges, $d$ double intersections, and a running count sum of $C$, we have:
\begin{equation}\label{Edge Formula}
d+e = \frac{C-n}{2}.
\end{equation}
\end{theorem}

\begin{proof} We proceed by induction on $n$.

Base case: when $n=1$, there is one arc and no intersections at all, so $e=d=0$.  The running count sequence is $[1,0]$, hence $C=1$ and \eqref{Edge Formula} holds.

Now suppose that \eqref{Edge Formula} holds for $n=k$.  Consider a collection of $k+1$ circular arcs. Pick an arc $A_1=[x_1,y_1]$.  We reduce this collection of $k+1$ arcs to a collection of $k$ arcs in two steps, showing that both sides of \eqref{Edge Formula} are changed by the same amount in each step.

Step 1: Let $p$ be the number of left endpoints contained in $A_1$ (other than $x_1$).  Move $y_1$ counterclockwise about the circle until there are no endpoints between it and $x_1$.  Call this new arc $A'_1=[x_1,y'_1]$.  

Consider the left hand side of \eqref{Edge Formula}.  When $y_1$ passes a left endpoint, then the two arcs no longer intersect at that left endpoint.  However, the two arcs may still intersect at their other endpoints.  If we began with a double intersection, then they now only intersect once, and $d$ decreases by 1.  If we began with a single intersection, then they no longer intersect, so $e$ decreases by 1.  When $y_1$ passes a right endpoint, there is no change in $d$ or $e$ as can be seen by inspection.  Thus, the left hand side of \eqref{Edge Formula} decreases by $p$.

For each left endpoint that $y_1$ passes on the way, $C$ is decreased by 2.  To see this, consider the running counts at the two endpoints.  When the overlap of their arcs is removed, each endpoint intersects one fewer arc.  Hence, the running count of each endpoint decreases by one.  For each right endpoint that $y_1$ passes, $C$ remains unchanged.  Thus, in total, $C$ decreases by $2p$ and the right hand side of \eqref{Edge Formula} decreases by $p$.

Step 2: Remove $A'_1$ from the collection of arcs.  As $A'_1$ is arbitrarily small, $A'_1$ intersects $A_j$ if and only if $x_1\in A_j$.  Thus, the decrease in the number of edges is equal to the running count at $x_1$ minus 1, which we call $r$.  (Observe that the running count at $y'_1$ is $r$.)  There is no change in $d$ because no other endpoints are contained in $A'_1$.  Therefore, the left hand side of \eqref{Edge Formula} decreases by $r$.

Since the removed arc doesn't intersect at any other endpoints, none of the running counts are changed, except for the two corresponding to this removed arc, which are removed from the sum.  Thus, $C$ decreases by $r+(r+1)=2r+1$, and since $n$ decreases by 1, the right hand side of \eqref{Edge Formula} decreases by $r$.

Since \eqref{Edge Formula} holds for $n=k$ and both steps preserve equality, \eqref{Edge Formula} holds for $n=k+1$, and by induction it is established for all $n$.
\end{proof}

\section{Maximizing edges by maximizing the running count sum}

Given conditions $M,m,n$, we wish to construct collections of arcs which maximize edges or double intersections.  As a first (simplifying) step, we show that we may restrict our attention to collections of arcs with a maximal $LR$-sequence (Theorem~\ref{maxatmax}).  The following result provides a maximal $LR$-sequence which is unique up to the choice of starting point.

\begin{proposition}\label{maxlr}
Fix parameters $M,m$, and $n$.  Then a maximal running count sequence is:
\[[m+1,m+2,\ldots,M-2,M-1,M,M-1,M,\ldots,M,M-1,M-2,\ldots,m+1,m]\]
and its running count sum is
\[C_{max} = (M+m)(M-m) + (2M-1)(n-M+m).\]
\end{proposition}
\begin{proof}
Without loss of generality, suppose that the minimum agreement number is the last number in a running count sequence $[r_i]$ (i.e. $r_{2n}=m$).  We maximize the running count sum by maximizing the value of each position in the sequence.  The highest possible value of $r_1$ is $m+1$, the highest possible value of $r_2$ is $m+2$, and so on until we get $r_{M-m}=M$.  Since the maximum agreement number is $M$, and the running count must increase or decrease by 1 at each endpoint, the running counts must alternate between $M$ and $M-1$ until the last $M-m$ positions in the sequence.  The running counts here must decrease by 1 at each position so that $r_{2n}=m$, the minimum agreement number.  This shows the maximal running count sequence and $C_{max}$ are as above.
\end{proof}

There are other maximal running count sequences, but they only vary based on the placement of the minimum agreement number $m$.  Since the minimum agreement number occurs just once in each sequence, we adopt the convention that it must occur in the last position and from here on refer to the sequence described in Proposition~\ref{maxlr} as \emph{the} maximal running count sequence.

Likewise, \emph{the} maximal $LR$-sequence is the $LR$-sequence which is derived from the maximal running count sequence.  This is possible because the running count sequence increases by 1 at left endpoints and decreases by 1 at right endpoints.  From Proposition~\ref{maxlr}, we have that the maximal $LR$-sequence will begin with $M-m$ many $L$'s, then the sequence $RL$ repeated $n-M+m$ times, then finally $M-m$ many $R$.  For example, if we fix $n=7, M=5,$ and $m=2$, the maximal running count sequence is $$[3,4,5,4,5,4,5,4,5,4,5,4,3,2],$$ and its $LR$-sequence is $$[L,L,L,R,L,R,L,R,L,R,L,R,R,R].$$


The following technical lemma will allow us to restrict our attention to collections of arcs with the maximal running count sequence 
(Theorem \ref{maxatmax}).

\begin{lemma}\label{maxEdgesDI}
Suppose $\mathcal{A}=\{A_i\}$ is a collection of $n$ arcs with maximum and minimum agreement numbers $M$ and $m$, $e$ edges, $d$ double intersections, and running count sum $C<C_{max}$.  Then, there exists another collection of $n$ arcs $\mathcal{A'}=\{A'_i\}$, also with maximum and minimum agreement numbers $M$ and $m$, with $e'$ edges, $d'$ double intersections, and running count sum $C'$ such that:
\begin{enumerate}
	\item[(i)] For all $i$, $A_i\subseteq A'_i$,
	\item[(ii)] $e\leq e'$,
	\item[(iii)] $d\leq d'$, and
	\item[(iv)] $C<C'$.
\end{enumerate}
\end{lemma}
We will prove this result by showing that if $C<C_{max}$, we can extend the given arcs to increase $C$ without changing $M$ or $m$.

\begin{proof}

Let $[q_i]$ and $[r_i]$ be the $LR$- and running count sequences of $\mathcal{A}$ respectively.  Without loss of generality, we assume $r_{2n}=m$.  Let the maximal $LR$-sequence be denoted $[q'_i]$ and the maximal running count sequence $[r'_i]$.  Note that $r_i\leq r'_i$ for all $i$.  Since $C<C_{max}$, we choose the smallest $j$ such that $r_j< r'_j$.  For this choice of $j$, $q_j=R$, $q'_j=L$, $r_j+2=r'_j\leq M$, and $r_{j-1}\leq M-1$.  (This is because $r_{j-1}=r'_{j-1}$ and the running count must change by $1$ or $-1$ at each position.)
\begin{figure}[h]
  \begin{center}
     \includegraphics[scale=0.45]{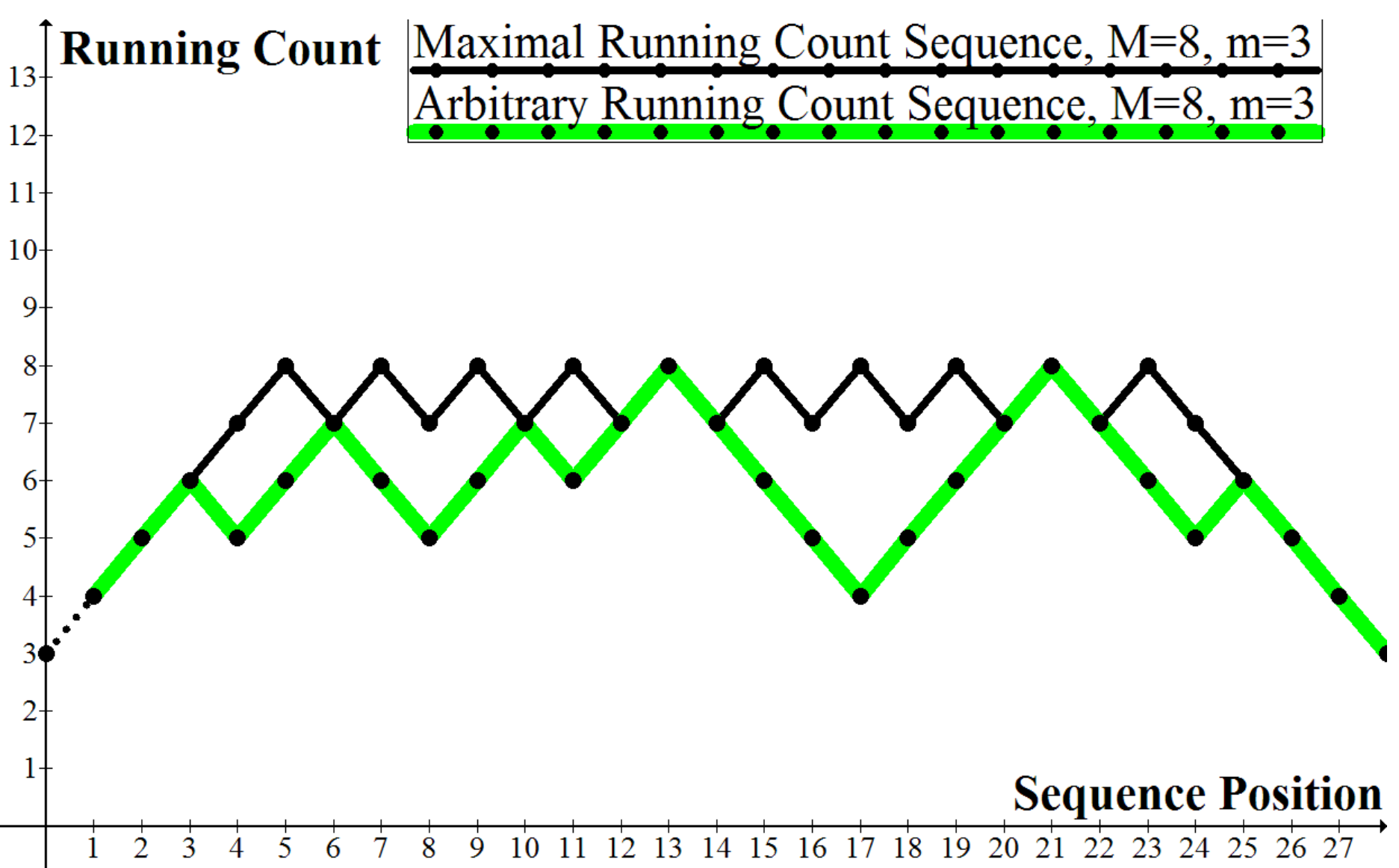}
  \end{center}
  \caption{\small{A graphical representation of the maximal running count sequence with $M=8, m=3, n=14$ and an arbitrary running count sequence with the same conditions.  For the arbitrary sequence, the first $j$ for which $r_j<r'_j$ is $j=4$.}}
  \label{Graph7}
\end{figure}

In order to obtain $\mathcal{A'}$, we will move some endpoint around the circle to extend an arc while maintaining the same minimum and maximum agreement numbers.  Our choice of endpoint depends on the relation between the $L$'s and $R$'s to the right of $q_j$.

Case 1: Suppose that all $L$'s occurring after $q_j$ are paired with an $R$ immediately before them.  We note this implies that there are no instances of consecutive $L$'s in this part of the sequence.  Thus, $r_k\leq r_{j-1}\leq M-1$ for all $k\geq j$.

Which endpoint we move depends on whether $q_{j-1}$ is an $R$ or an $L$.  If $q_{j-1}=R$, then, since $q_j=R$, $q_{j-1}$ is not paired with an $L$ that occurs to the right of it in the $LR$-sequence.  For this arc, move its right endpoint clockwise around the circle until it passes another arc's left endpoint, increasing $C$ by 2.  The maximum agreement number is still $M$ because the running count of each endpoint only increases by 1.

On the other hand, suppose that $q_{j-1}=L$.  Note that $q_j=R$, but $q'_j=L$, so the number of $L$'s at the beginning of $[q_i]$ is less than $M-m$.  Thus, the maximum agreement number $M$ is never reached, so this situation is impossible.

Case 2: Suppose there exists an $L$ occurring after $q_j$ which is not paired with an $R$ immediately before it.  Choose the minimum $k>j$ for which $q_k$ is such an $L$ and call its endpoint on the circle $x$.

If $q_{k-1}=R$, then move $x$ counterclockwise around the circle until it passes the right endpoint corresponding to $q_{k-1}$, say, $y$.  The running counts for $x$ and $y$ both increase by 1, so $C$ increases by 2.

\begin{figure}[h]
  \begin{center}
     \includegraphics[scale=0.45]{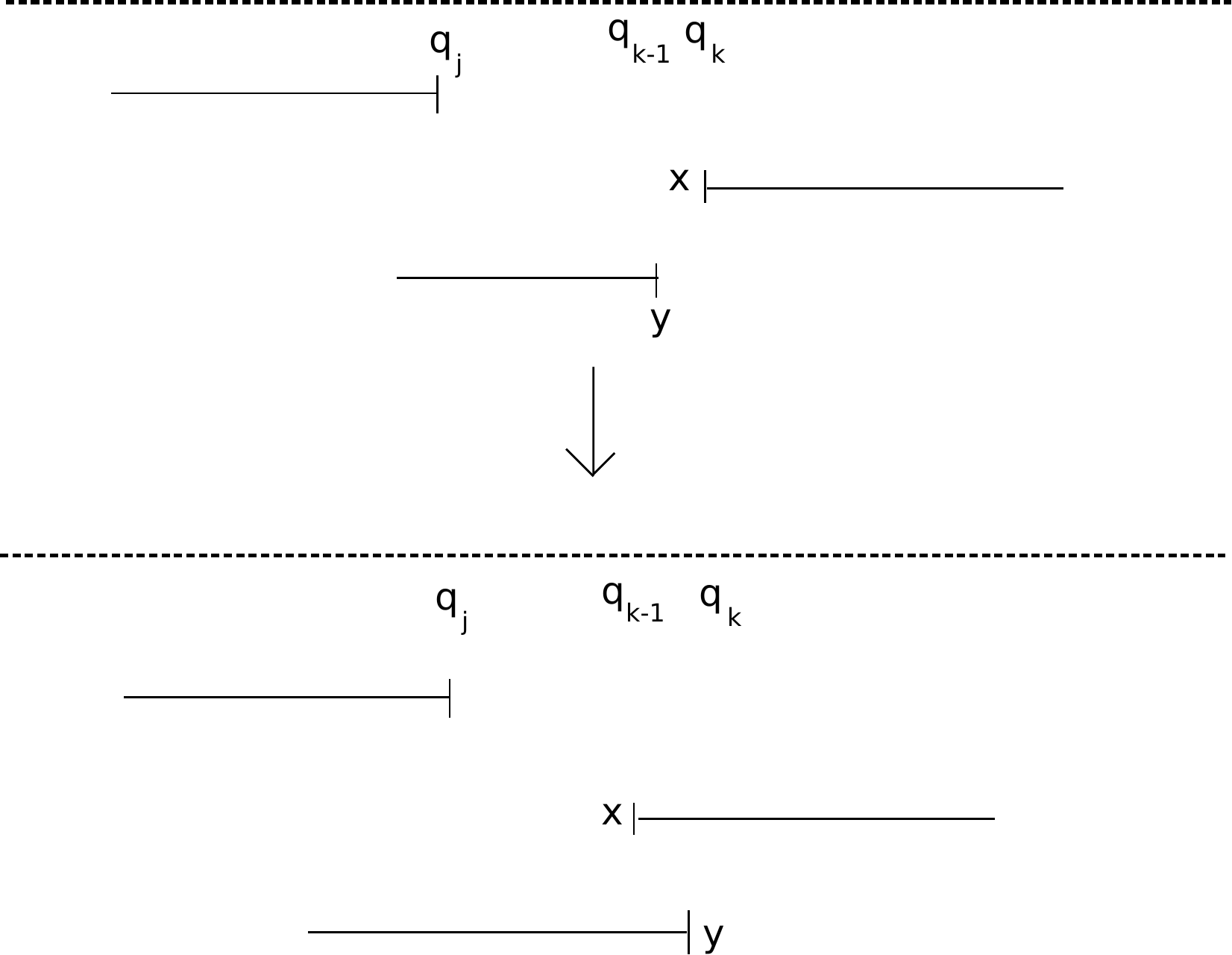}
  \end{center}
  \caption{\small{An illustration of the operation performed in Case 2 when $q_{k-1}=R$.}}
  \label{case2}
\end{figure}

If $q_{k-1}=L$, then it must be paired with $q_{k-2}=R$.  Move $x$ counterclockwise until it passes the right endpoint corresponding to $q_{k-2}$.  Again, the running counts for the two endpoints each increase by 1, so $C$ increases by 2.

For both possible $q_{k-1}$'s, it can be verified that $M$ is still the maximum agreement number because of the condition that $k$ is minimum, and that all $L$'s in between $q_j$ and $q_k$ are paired with adjacent $R$'s.

In all cases, we are only extending arcs.  Thus, condition $(i)$ is established, and $(ii)$ and $(iii)$ are implied by $(i)$. The minimum agreement number is still $m$ because no arcs are extended through the point of minimum agreement on the circle.

Therefore, for any collection of arcs satisfying conditions $M,m$, and $n$, it is possible to construct a new collection of arcs from the old collection satisfying the same conditions.
\end{proof}

\begin{theorem}\label{maxatmax}
Given $M,m$ and $n$, the maximum number of edges obtainable for a collection of arcs can be obtained by a collection of arcs with maximal running count sum.  The maximum number of double intersections is also obtainable by a (possibly different) collection of arcs with the maximal running count sum.
\end{theorem}
\begin{proof}
Let $\mathcal{A}$ be a collection of arcs with the maximum number of edges and let its running count sum be $C$.  If $C=C_{max}$, then we are done.  If $C<C_{max}$, then apply Lemma~\ref{maxEdgesDI} as many times as necessary to obtain a new collection of arcs $\mathcal{A'}$ with running count sum $C_{max}$.  By Lemma~\ref{maxEdgesDI}, $\mathcal{A'}$ must have at least as many edges in its intersection graph as $\mathcal{A}$.  Therefore, $\mathcal{A'}$ has the maximum number of edges, completing the proof.  The proof for double intersections is analogous.
\end{proof}

\section{Arcs with maximal edge count and minimal double intersections}

The preceding section showed that for fixed $M,m,n$, there existed a collection of arcs with the maximal $LR$-sequence that had the maximum number of edges. We now provide such a collection, $\A_{max}$, and show that $\A_{max}$ has the greatest edge count for collections of arcs satisfying $M,m,n$.

Using this collection of arcs we will show in Theorem~\ref{discreteEdgeCount} that if a collection of $n$ arcs $\mathcal{A}$ has at least $e_{min}(M,m,n)$ edges, and every point on the circle is covered by at least $m$ arcs, then some point is covered by at least $M$ arcs.  A continuous result also holds for large $n$, where $m=\gamma M$ and $M=\beta n$ (Theorem~\ref{continuousEdgeCount}).


We now define the collections of arcs with the maximum edge count for a fixed $M,m,n$:

\begin{defn}
We construct a collection of arcs, $\mathcal{A}_{max}$, for fixed $M,m,n$ with the maximal $LR$-sequence.  Number the $L$'s $L_1,\ldots,L_n$ from left to right, and do the same for the $R$'s.  Assign the endpoints to two classes of arcs: $A = \{A_1,\ldots,A_m\}$ and $B = \{B_{m+1},\ldots,B_n\}$, where $A_i$ has endpoints $R_i$ and $L_{n-m+i}$, and $B_{m+i}$ has endpoints $L_i$ and $R_{m+i}$.  (See Figure~\ref{myintervalsfig}).
\end{defn}
We will refer to the arcs in class $A$ as A-type arcs, and the arcs in class $B$ as B-type arcs.  In a general collection of arcs $\A$ satisfying $M,m,n$, A-type arcs will be those whose left endpoints occur to the right of their right endpoints in an $LR$-sequence ending with a running count of $m$, and the other arcs will be B-type arcs.  Note that the A-type arcs are the $m$ arcs that intersect at a point of minimum agreement between the last member of the $LR$-sequence and the first (the point of ``wrap-around'' in Fig.~\ref{myintervalsfig}).

It is of some interest that no arc in $\A_{max}$ properly contains another; its graph is therefore known as a \emph{proper circular arc graph}.  Collections with this property are also used in \cite{hardin}.
\begin{theorem}\label{myintervals}
This collection of arcs $\mathcal{A}_{max}$ maximizes the edge count among all collections of arcs satisfying conditions $M,m,n$.
\end{theorem}
\begin{figure}[h]
  \begin{center}
     \includegraphics{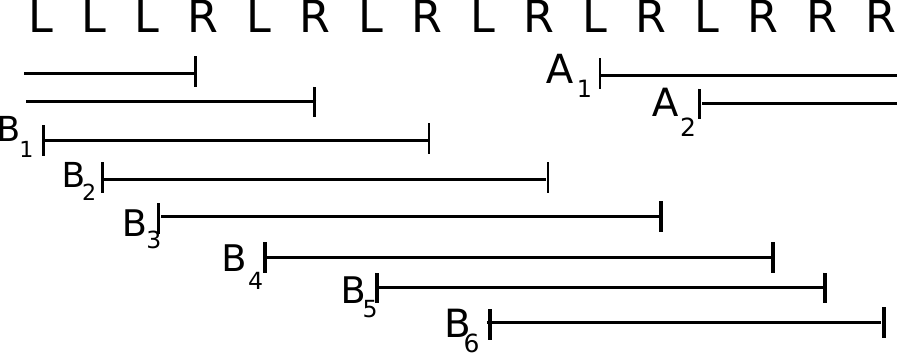}
  \end{center}
  \caption{$\mathcal{A}_{max}$ for $n=8,M=5,m=2$.  The $A_i$ intervals ``wrap around'' the other side of the circle in this $LR$-sequence, and there are $m=2$ of them.  Note that there are two double intersections: $A_1$ with $B_3$, and $A_2$ with $B_4$.}
  \label{myintervalsfig}
\end{figure}

We begin our proof of Theorem~\ref{myintervals} by counting the number of double intersections in $\A_{max}$.  We will then show that $\A_{max}$ has the minimum number of double intersections out of all collections of arcs satisfying conditions $M,m,n$ and having the maximum $LR$-sequence.  By Theorem~\ref{edgeformula} and Theorem~\ref{maxatmax} this will show that $\A_{max}$ has the maximum number of edges for all collections of arcs satisfying $M,m,n$.

\subsection{Counting double intersections in $A_{max}$}
Theorem~\ref{alledges} will show that when $M+m\geq n+1$, $\mathcal{A}_{max}$ has all possible edges.  For now we restrict our attention to conditions satisfying $M+m<n+1$.  We count the double intersections as follows.  First, note that B-type arcs can never doubly intersect each other, because doubly intersecting arcs must cover the whole circle, and no B-type arc covers the point of least intersection.

Two A-type arcs $A_i$ and $A_j$ will doubly intersect when the right endpoint of $A_i$ is to the right of the left endpoint of $A_j$ in the $LR$ sequence.  So as long as every A-type right endpoint is to the left of every A-type left endpoint, no members of $A$ will doubly intersect. (See Figure~\ref{myintervalsfig} for an example.)  This occurs when $m+M \le n+1$.  
To see this, consider the sequence of alternating $L$ and $R$ endpoints in the middle of the maximal $LR$-sequence.  It has $n-(M-m)$ $L$ endpoints and the same number of $R$'s.  If the $m$ A-type right endpoints cover more than half of these, then some of them will lie to the right of A-type left endpoints.  So we require that $2m \le n-(M-m)+1$, or $M+m\le n+1$.  Since we only consider $M,m,n$ that meet this requirement, we must only count double intersections between A-type and B-type arcs.

\begin{theorem}\label{dmin}
The collection of arcs $\mathcal{A}_{max}$, with $M+m\le n+1$, will have $d$ double intersections, where
\[
  d= \left\{
  \begin{array}{l l}
    0 &  \quad 2M \le n+1\\
    m(2M-n-1) & \quad 2M > n+1.\\
  \end{array} \right.
\]
\end{theorem}
\begin{proof}
This result will be achieved by counting the number of double intersections involving each A-type arc.  Since A-type arcs will only doubly intersect B-type arcs, this will not double count any double intersections.

Consider an arc $A_i$ with endpoints corresponding to $L_{n-m+i}$ and $R_i$.  This arc will doubly intersect a B-type arc $B_i$ if $L_{B_i}$ is to the left of $R_{A_i}$ and $R_{B_i}$ is to the right of $L_{A_i}$.  The count of $L$ endpoints of B-type arcs to the left of $L_{n-m+i}$ is $M-m+i-1$.  The number of double intersections involving $A_i$ will be the number of these $L$ endpoints that are paired with an $R$ to the right of $L_{A_i}$.  There are $(n-M-m+i)$ $R$ endpoints of B-type arcs to the left of $L_{A_i}$, so the number of double intersections involving $A_i$ is
\[d_{A_i} = 2M-n-1.\]
Summing over all the $m$ A-type arcs, this gives the double intersection count
\[d=m(2M-n-1).\]
This count is negative for $2M \le n+1$ because no double intersections will occur in that case.
\end{proof}

\subsection{$\A_{max}$ has the maximal edge count.}
We will now show that $\A_{max}$ has the minimum number of double intersections satisfying conditions $M,m,n$ with the maximal $LR$-sequence.  We will prove a formula in Theorem \ref{DIformula} that, given a collection of arcs $\A$, gives a lower bound for the number of double intersections in $\A$.  We will then show in the proof of Theorem \ref{myintervals} that $\A_{max}$
\begin{enumerate}
\item achieves this lower bound, and
\item minimizes the value of the lower bound among families of arcs satisfying conditions $M,m,n$ and the maximal $LR$-sequence.
\end{enumerate}

\begin{defn}
Let $A_i$ denote one of the A-type arcs in $\A$, with endpoints $L_{A_i}$ and $R_{A_i}$.  Then let $l_{A_i}$ denote the number of left endpoints to the left of $R_{A_i}$ in the $LR$-sequence and let $r_{A_i}$ denote the number of right endpoints to the left of $L_{A_i}$.
\end{defn}
For example, in Figure~\ref{myintervalsfig}, $l_{A_1} = 3$ and $r_{A_1} = 4$.

\begin{theorem}\label{DIformula}
Let $\A$ be a collection of arcs satisfying conditions $M,m,n$, and say that $d_{\A}$ is the number of double intersections in $\A$.  Then,
\[d_{\A} \ge \sum_{i=1}^m (l_{A_i} - r_{A_i} + m).\]
\end{theorem}
\begin{proof}
This bound will be shown by counting the number of double intersections involving each A-type arc.  Recall that a B-type arc $B_k$ will doubly intersect $A_i$ if $L_{B_k}$ is to the left of $R_{A_i}$ and $R_{B_k}$ is to the right of $L_{A_i}$ (see arcs $A_1$ and $B_3$ in Figure~\ref{myintervalsfig}).  Then there are $l_{A_i}$ left endpoints that may belong to such a $B_k$. (If any of these left endpoints belongs to an A-type arc then that arc doubly intersects with $A_i$.  We ignore such double intersections for now.)  However, if these endpoints are paired with any of the $r_{A_i}$ right endpoints to the left of $L_{A_i}$ their arcs will not doubly intersect $A_i$.  The figure $l_{A_i} - r_{A_i}$ represents the number of left endpoints that cannot be paired off in this way, so at least that number of B-type arcs must doubly intersect $A_i$.

The number of double intersections involving $A_i$ is still undercounted.  We will show that it is undercounted by at least $m$.  There may be right endpoints counted in the figure $r_{A_i}$ that belong to A-type arcs.  These endpoints cannot be paired with the left endpoints counted in $l_{A_i}$, so they should not be subtracted from the double intersection count.  If no A-type arc $A_j$ doubly intersects $A_i$, then all of the $m$ A-type right endpoints lie to the left of $L_{A_i}$. (This is the case for both A-type arcs in Figure~\ref{myintervalsfig}).  Then the count of double intersections involving $A_i$ should increase by $m$.  Suppose instead that $R_{A_j}$ lies to the right of $L_{A_i}$.  Then although one fewer B-type arc must doubly intersect $A_i$, $A_j$ doubly intersects $A_i$.  Thus the number of double intersections involving $A_i$ does not change, and is still at least $l_{A_i} - r_{A_i} + m$.

Note that this count of double intersections does not count a double intersection between two A-type arcs twice.  The count for arc $A_i$ only counts double intersections with A-type arcs whose right endpoints are to the right of $L_{A_i}$, not those with left endpoints to the left of $R_{A_i}$.

Summing this count over all the A-type arcs gives the formula above.
\end{proof}

\begin{proof}[Proof of Theorem~\ref{myintervals}]
Consider the bound from Theorem	~\ref{DIformula} for $\A_{max}$.  Recall that in $\A_{max}$, the arc $A_i$ has the endpoints $R_i$ and $L_{n-m+i}$.  Then $l_{A_i} = M+m+i-1$ and $r_{A_i} = n-M+i$.  Then
\[\sum_{i=1}^m (l_{A_i} - r_{A_i} + m) = \sum_{i=1}^m (2M-n-1) = m(2M-n-1).\]
This sum is equal to $d_{\A_{max}}$ as shown in Theorem~\ref{dmin}, so the inequality of Theorem~\ref{DIformula} is an equality for $\A_{max}$.

Because the right endpoints of the A-type arcs are the $m$ leftmost right endpoints, $\sum l_{A_i}$ is minimized.  Because the left endpoints of the A-type arcs are the $m$ rightmost, $\sum r_{A_i}$ is maximized.  So the difference of these sums must be minimized, minimizing the right-hand side of Theorem~\ref{DIformula}.

Since $\A_{max}$ achieves the lower bound in Theorem~\ref{DIformula}, it has the minimum number of double intersections of all arcs satisfying $M,m,n$ with the maximal $LR$-sequence, and thus has the maximum number of edges of all arcs satisfying $M,m,n$.
\end{proof}

\subsection{The minimum edge count to guarantee agreement M}
Using Theorem~\ref{dmin}, we can state exactly the maximum number of edges of a collection of $n$ circular arcs with maximum and minimum agreement numbers $M$ and $m$.  Recall from the Edge Formula that
\[d+e = \frac{C- n}{2}.\]
Since $C$, $n$ and the minimum value for $d$ are known, the maximum edge count can be calculated.  This yields the following theorems: 

\begin{theorem}\label{alledges}
The collection of arcs $\mathcal{A}_{max}$ for $M,m,n$ where $M+m\geq n+1$ has $n\choose 2$ edges; that is, all possible edges are present.
\end{theorem}
\begin{proof}
Observe that when $m+M=n+1$, $\mathcal{A}_{max}$ has $n\choose 2$ edges (by Theorem~\ref{discreteEdgeCount} and the Edge Formula).  When $m$ is increased by 1, another arc in the collection becomes an A-type arc.  It can be shown that this new arc properly contains the B-type arc it replaces, so no intersections between the arcs are lost.  So it may be shown by induction that $\mathcal{A}_{max}$ has $n\choose2$ edges whenever $M+m\geq n+1$.
\end{proof}

We now present one of our main results:
\begin{theorem}\label{discreteEdgeCount}
For any $M,m,n$ where $M+m \le n+1$ the collection $\A_{max}$ has $e_{max}$ edges, and no family of arcs satisfying $M,m,n$ can have more than $e_{max}$ edges. Then:
\[
  e_{max} = \left\{
   \begin{array}{l l}
   \frac{1}{2}(M+2Mn+2Mm-M^2-m^2-m-2n) & \quad 2M \le n+1\\
   \frac{1}{2}(M+2Mn-2Mm-M^2-m^2+2mn+m-2n) & \quad 2M \ge n+1.\\
  \end{array} \right.
\]
Then the minimum edge count needed to guarantee an agreement $M$ with fixed $m$ and $n$ is as follows (when $M+m\le n+1$):
\[
  e_{min} = \left\{
  \begin{array}{l l}
    \frac{3}{2}M+Mn+Mm-\frac{1}{2}M^2 - \frac{3}{2}m - \frac{1}{2}m^2 - 2n &  \quad 2M \le n+3\\
    \frac{3}{2}M + Mn -Mm -\frac{1}{2}M^2 +m +mn - \frac{1}{2}m^2 - 2n & \quad 2M > n+3.\\
  \end{array} \right.
\]
\end{theorem}
\begin{proof}
The maximality of $\A_{max}$ was shown in Theorem~\ref{myintervals}.  The $e_{max}$ equations come from counting the edges in $\mathcal{A}_{max}$.  The edge formula may be combined with Theorem~\ref{maxlr} and the known number of double intersections in $\mathcal{A}_{max}$ to yield this edge count, which we have shown to be the greatest possible for $M,m,n$.

Note that no family of $n$ arcs with minimum agreement number $m$ can have more than $e_{max}$ edges without having a maximum agreement of at least $M+1$.  Then $$e_{min}(M) = e_{max}(M-1)+1.$$ 
The $e_{min}$ equations are thus derived by substituting $M-1$ for $M$ in the $e_{max}$ equations.  Since $e_{max} = {n\choose2}$ when $M+m=n+1$, there does not exist an $e_{min}$ which can guarantee any greater agreement number in this situation.
\end{proof}

Letting $n \rightarrow \infty$, we obtain this asymptotic result:
\begin{theorem}\label{continuousEdgeCount}
Let $\alpha,\beta,\gamma$ be defined so that $e_{min} = \alpha {n \choose 2}$, $M = \beta n$, and $m = \gamma M$. 
In the limit as $n \rightarrow \infty$ we obtain:
\[
  \alpha = \left\{
  \begin{array}{l l}
    \beta(2-(1-\gamma)^2\beta) &  \quad 0 \le \beta \le \frac{1}{2}\\
    \beta(\gamma+1)(2-\beta(\gamma+1)) & \quad \frac{1}{2} < \beta \le \frac{1}{1+\gamma}.\\
  \end{array} \right.
\]
\end{theorem}
See Figure~\ref{Graph5}.  Note that $\alpha$ is the fraction of all possible edges, $\beta$ is the agreement proportion, and $\gamma$ is the ratio of minimum to maximum agreement.  So Theorem~\ref{continuousEdgeCount} shows the minimum fraction of edges $\alpha$ needed to guarantee an agreement proportion $\beta$.

Note that when $\beta = \frac{1}{2}$, the two formulas both give $\alpha = \frac{1}{4}\gamma^2 + \frac{1}{2}\gamma + \frac{3}{4}$, so our function is continuous.  Our formula is defined for $m+M\le n+1$, or in the continuous case, $\beta(1+\gamma) \le 1$.  When $\beta(1+\gamma) = 1$, this gives $\alpha = 1$, since all possible edges must be present to guarantee an agreement ratio of $\beta$.

When $\gamma = 0$, that is when there is a point on the circle that is not part of any arc, the collection of arcs is identical to a collection of intervals on a straight line.  Our formula then becomes 
\[\alpha = \beta(2-\beta),\]
which is equivalent to the result for interval graphs published by Abbott and Katchalski \cite{abbott-katchalski}.  They proved their result using the chordal property of interval graphs to show that an agreement proportion of at least $\beta$ was achievable, and constructed a set of intervals to show that no greater agreement proportion can be guaranteed.  Our methods differ in relying more on the properties of the arcs themselves, and deriving results about the intersection graphs as a corollary.

Our result also allows an upper bound for $e_{min}$ to be calculated when $m$ or $\gamma$ is unknown.  In this case, we can set $m=M-1$ and $\gamma=1$.  The $e_{min}$ will then be greater than or equal to the value calculated for a known $m$ or $\gamma$, and a lower bound on the maximum agreement of $M$ will be guaranteed.  The bound in this case will be $\alpha = 2\beta$ (see Figure~\ref{Graph5}).

\begin{figure}[h]
  \begin{center}
     \includegraphics[scale=0.6]{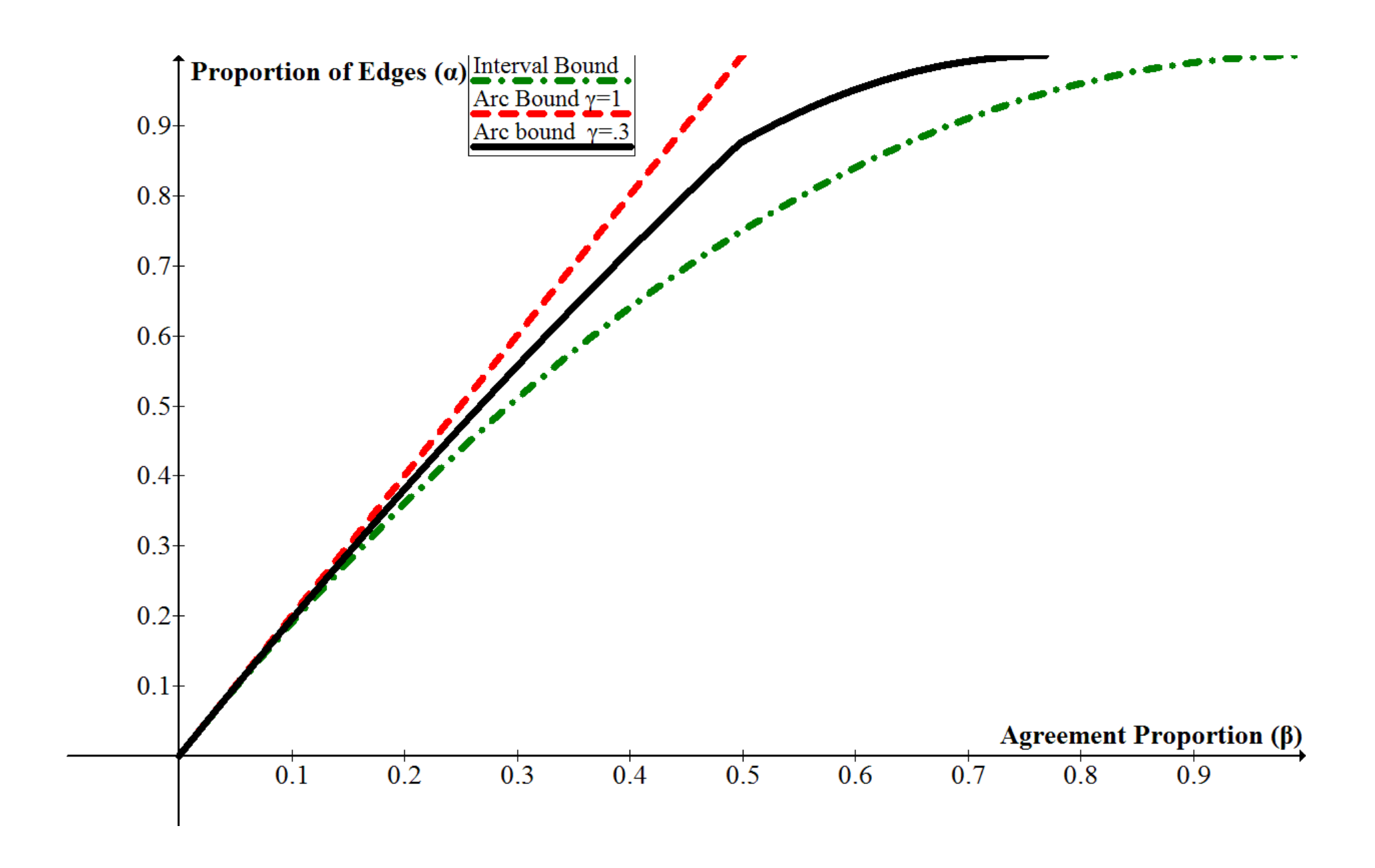}
  \end{center}
  \caption{Theorem~\ref{continuousEdgeCount} illustrated for several values of $\gamma$.} 
  \label{Graph5}
\end{figure}

\section{Characterizing allowable collections by maximizing double intersections}

In the previous section, we derived an expression for the proportion $\alpha$ of edges needed to guarantee an agreement proportion $\beta$ with fixed ratio $\gamma = m/M$.  In this section, we determine $\alpha$ given $\beta$, $\gamma$ and information about $d$, then determine when collections of arcs satisfying these conditions exist.

We describe this information as the proportion $p$ of total possible intersections which are double intersections. That is
\begin{equation}
p=\frac{d}{{n\choose 2}}
\end{equation}
Given $M,m,n$, and $p$ (which defines $d$), it is clear from the Edge Formula that edges will be maximized when the running count is maximized.  However, it is possible that the specific value of $p$ is not realizable by a collection of arcs satisfying conditions $M,m$, and $n$, but for now we will assume that it is, then spend the rest of this section discussing when the value is allowable.

Assuming the running count sum $C_{max}$, the Edge Formula and Proposition~\ref{maxlr} give us:
\begin{equation}
e_{min}=\frac{-M^2+2nM+2mM+M-m^2-2n-m}{2}-p{n\choose2}
\end{equation}
which for asymptotically large $n$, reduces to
\begin{equation}
\alpha=\beta (2-(1-\gamma)^2 \beta)-p
\end{equation}
This holds \emph{assuming there exists a collection of arcs which satisfies the conditions $M,m,n,$ and $p$}.  As we saw in the previous section, not all combinations of $M,m,n$, and $d$ are realizable as a collection of arcs.  If we use this equation to plot proportion of edges versus agreement proportion (as in Figure~\ref{Graph6}), then we observe that for small values of $\beta$, fewer edges are required to guarantee a particular agreement proportion for a collection of arcs than for a collection of intervals.
The goal of the remainder of this section is to show that the bound derived above only holds for large enough $\beta$.  Here, by ``large enough," we mean that the $\alpha$ required to guarantee $\beta$  is greater for collections of arcs than for collections of intervals (see Remark~\ref{theremark}).

\begin{theorem}\label{dmax}
Let $\mathcal{A}$ be a collection of arcs satisfying conditions $M$, $m$, $n$.  Then $\mathcal{A}$ can have at most
\[d_{max}=\frac{m}{2}(2M-m-1)\]
double intersections.
\end{theorem}
\begin{proof}
We prove Theorem~\ref{dmax} by first describing a collection of arcs $\mathcal{A}$ with $d_{max}$ double intersections.  Then we suppose there exists another collection of arcs $\mathcal{A'}$ with more than $d_{max}$ double intersections and find a contradiction via a pigeonhole argument.

Let $\mathcal{A}$ be a collection of arcs with maximal running count sequence such that the left endpoint of each A-type arc immediately follows its right endpoint in the $LR$-sequence (see Figure~\ref{myintervalsDI}); so, for each A-type arc $A_i$, there are no endpoints of any arc in the complement of $A_i$.  Therefore, an arc doubly intersects $A_i$ if and only if it intersects at the right endpoint of $A_i$.  The number of such arcs is the running count at the right endpoint.  Recall that A-type arcs must have their right endpoint occur to the left of their left endpoint in the $LR$-sequence.  Thus, it can be seen by inspection that the endpoints of the A-types must occur during the part of the $LR$-sequence where the running count is maximal.  Therefore, there are $(M-1)$ arcs intersecting the right endpoint of each A-type arc, meaning that each A-type arc doubly intersects $M-1$ arcs.

\begin{figure}[h]
  \begin{center}
     \includegraphics{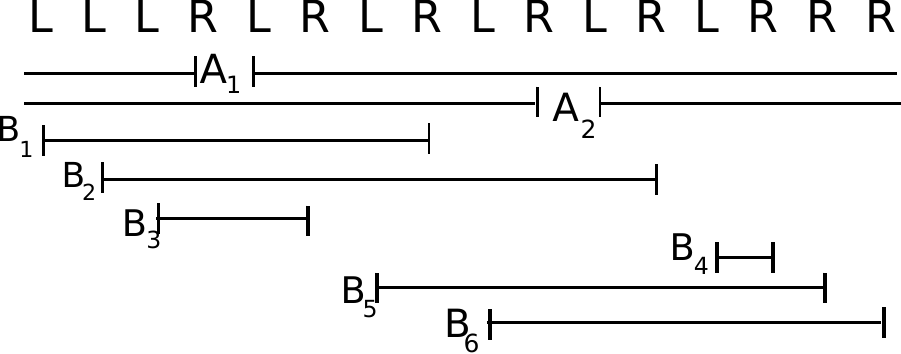}
  \end{center}
  \caption{A collection of arcs where the left endpoint of each A-type arc immediately follows its right endpoint in the $LR$-sequence, as with $\mathcal{A}$ in the proof of Theorem~\ref{dmax}.  Here, $n=8$, $M=5$, and $m=2$.}
\label{myintervalsDI}
\end{figure}

It is clear that each of the A-type arcs pairwise doubly intersect, so by adjusting for those double intersections which are double counted, we have that $\mathcal{A}$ contains exactly
\begin{equation}\label{ijustwannareference}
m(M-1)-{m \choose 2}
\end{equation}
double intersections.  This number produces the $d_{max}$ in the theorem statement.

Now suppose there exists a collection of arcs $\mathcal{A'}$ with more than $m(M-1)-{m \choose 2}$ double intersections.  For each endpoint of an A-type arc $A'_i$, count the number of arcs other than $A'_i$ intersecting that endpoint.  Let $D^*$ be the sum of these counts over all endpoints of all A-type arcs.  We claim that
\[D^*\geq 2(d+{m \choose 2}).\]
Note that each of the A-type arcs cover a point of minimum agreement on the circle and thus pairwise intersect.  For each pair of A-type arcs, at least two of their four endpoints are contained in the other arc, so $D^*\geq 2{m\choose 2}$.  Now consider a double intersection.  If it is between two A-type arcs, then each of the four endpoints is contained in the other arc.  But, two of these containments are already accounted for, so the total contribution to $D^*$ is 2.  If the double intersection is between an A-type arc and a B-type arc, then the endpoints of the A-type arc are contained in the B-type arc so the contribution to $D^*$ is also 2, proving our claim.

Now suppose $d>m(M-1)-{m \choose 2}$, that is, there are more double intersections than what we claim to be the maximum.  Then, $D^*>2m(M-1)$.  Since there are $2m$ endpoints of A-type arcs, the pigeonhole principle states that there must exist an A-type arc $A'_i$ with an endpoint contained in at least $M$ arcs other than $A'_i$.  When we take into account $A'_i$ itself, there are at least $M+1$ arcs containing this endpoint, which contradicts that $M$ is the maximum agreement number.  Therefore, no such collection $\mathcal{A'}$ exists and $d_{max}=\frac{m}{2}(2M-m-1)$.
\end{proof}

From Theorem~\ref{dmax}, we can define a quantity $p_{max}=\frac{d_{max}}{{n\choose2}}$, which is the maximum proportion of double intersections realizable under those conditions.
\begin{corollary}\label{pmax}
For large $n$,
\begin{equation}
p_{max}(\beta,\gamma)=\beta^2 \gamma(2-\gamma)
\end{equation}
\end{corollary}
We will write $p_{max}(\beta)$ when $\gamma$ is fixed or obvious.

If we are given a fixed $p$, like at the beginning of this section, then we ask ourselves at what values of $\beta$ the resulting bound applies.

\begin{remark}\label{theremark}
For any $\beta$, there is a collection of arcs satisfying conditions $\beta, \gamma,$ and $p$ only if $p<p_{max}(\beta)$.  So we ask, for what $\beta^*$ is $p_{max}(\beta^*)=p$?  Because the $p_{max}$ function is strictly increasing with respect to $\beta$ for allowed values of $\beta$, this is equivalent to asking for which $\beta^*$ are the following conditions met:
\begin{enumerate}
\item If $\beta<\beta^*$, then $p_{max}(\beta)<p$ and the bound is invalid.
\item If $\beta>\beta^*$, then $p_{max}(\beta)>p$ and the bound is valid.
\end{enumerate}
From the explicit formula for the $p_{max}$ function, it is obvious that $\beta^*=\sqrt{\frac{p}{\gamma (2-\gamma)}}$.  It turns out that this $\beta^*$ is also the agreement proportion for which the bound determined at the beginning of this section intersects and surpasses the bound for collections of intervals.
\end{remark}


Summarizing our results, if the $\alpha$ required to guarantee a particular $\beta$ is less for collections of intervals than collections of arcs (given $p$), then there exists a collection of arcs satisfying conditions $\beta$, $\gamma$, $\alpha$, and $p$.  If $\alpha$ is greater for the interval bound than the arc bound, then no such collection exists.

\begin{figure}[t]
  \begin{center}
    \includegraphics[scale=0.6]{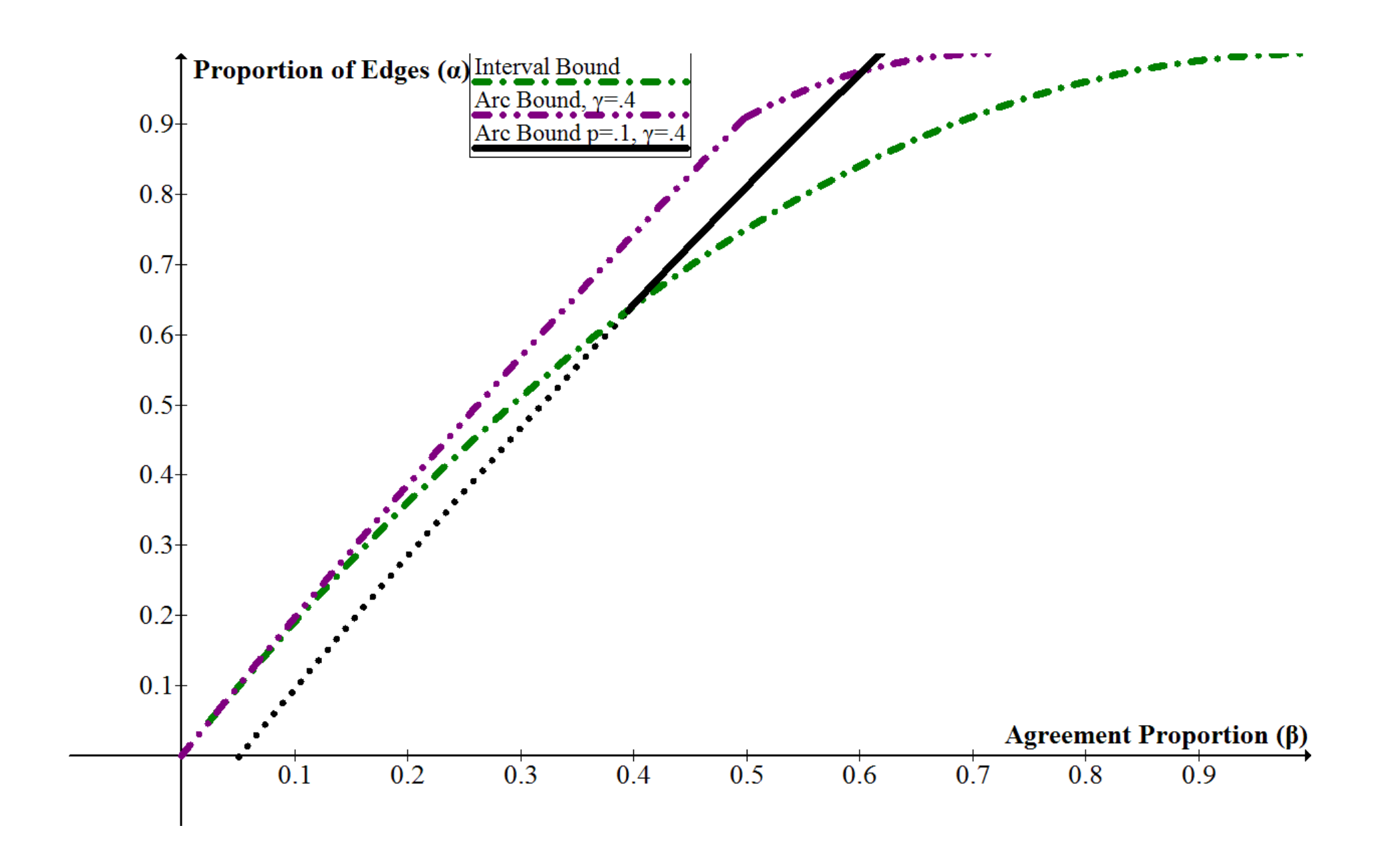}
  \end{center}
  \caption{The edge proportion $\alpha$ needed to guarantee agreement proportion $\beta$.  The dark line shows the bound obtained for $p=.1$.  The dotted line shows where this bound is invalid because the given values of $p,\beta$ and $\gamma$ are not simultaneously realizable.  The bound becomes valid where it passes the interval bound.}
  \label{Graph6}
\end{figure}

\section{Conclusion}
The combinatorial properties of collections of circular arcs were explored in several ways.  The agreement at the endpoints of arcs was related to a count of the intersections and double intersections (Theorem~\ref{edgeformula}).  For fixed minimum and maximum agreement, the maximum possible number of intersections (Theorem~\ref{discreteEdgeCount}) and double intersections (Theorem~\ref{dmax}) were found.  These bounds were used to determine the number of pariwise intersections necessary to guarantee a particular maximum agreement.  

In voting theory, voters may choose among a finite set of candidates on a circular spectrum.  Then it may be of interest to count the voters' agreement in the same way that we have in this paper.   Any voter's approval set consists of the set of candidates they approve of.  Then in a circular arc representation, each arc has its endpoints at two candidates, covering all the candidates which its voter supports.  This restriction on the family of arcs representing the voters may allow the bound relating pairwise agreement and the agreement number to be improved.  In fact, we have proved the following bound (but do not give the proof here):
\begin{proposition}
In a collection of $n$ circular arcs representing voter approval of $N$ candidates, if the maximum agreement is $M$, there are at most $e_{max}$ intersections among the arcs, where $e_{max}$ is defined by the following:
\begin{equation}
e_{max}\leq \frac{2nM-\frac{n^2}{N}-n}{2}.
\end{equation}
\end{proposition}
\begin{corollary}
For large $n$, the above bound reduces to
\begin{equation}
\alpha=2\beta-\frac{1}{N}.
\end{equation}
\end{corollary}

\subsection*{Acknowledgements}
We thank Michael Earnest and Emil Guliyev for careful reading of this paper.

\bibliographystyle{plain}
\bibliography{Bibliography}

%

\end{document}